\documentclass[11pt]{article}
\usepackage{amsmath,amsfonts,amsthm}
\usepackage{color}
\usepackage{mathrsfs,amsbsy}
\usepackage{dsfont}
\usepackage{hyperref}
\usepackage{graphicx}
\usepackage{amsmath,amsfonts,amsthm}
\usepackage{amssymb,latexsym}

\DeclareMathOperator*{\argmin}{argmin}
\DeclareMathOperator{\diag}{diag}

\DeclareMathOperator{\subspan}{span}

\newtheorem{proposition}{Proposition}[section]
\newtheorem{theorem}{Theorem}[section]
\newtheorem{lemma}{Lemma}[section]
\newtheorem{definition}{Definition}[section]

\newtheorem{remark}{{\sc Remark}}[section]
\newtheorem{example}{Example}[section]

\usepackage{algorithm,algorithmic}

\numberwithin{equation}{section}

\def\psdid{{PSD-{id}} }
\def\half{{\frac{1}{2}}}

\title{
Convergence analysis of a locally accelerated preconditioned steepest descent 
method for Hermitian-definite generalized eigenvalue problems}

\author{
Yunfeng Cai\thanks{
LMAM \& School of Mathematical Sciences,
Peking University, Beijing, 100871, China,
{\tt yfcai@math.pku.edu.cn}
} \quad
Zhaojun Bai\thanks{Department of
Computer Science and Department of Mathematics,
University of California, Davis 95616, USA,
{\tt bai@cs.ucdavis.edu}
} \quad
John E. Pask\thanks{
Condensed Matter and Materials Division,
Lawrence Livermore National Laboratory,
Livermore, CA 94550, USA,
{\tt pask1@llnl.gov}
} \quad
N. Sukumar\thanks{
Department of Civil and Environmental Engineering,
University of California, Davis 95616, USA,
{\tt nsukumar@ucdavis.edu}
}
}
\date{March 21, 2016}

\begin{document}

\maketitle

\begin{abstract}
By extending the classical analysis techniques due to 
Samokish, Faddeev and Faddeeva, and Longsine and McCormick among others, 
we prove the convergence 
of preconditioned steepest descent with implicit deflation (PSD-id) method
for solving Hermitian-definite generalized eigenvalue problems.  
Furthermore, we derive a nonasymptotic estimate of the rate of convergence
of the \psdid method.  We show that with the proper choice of the shift,
the indefinite shift-and-invert preconditioner
is a locally accelerated preconditioner, and is
asymptotically optimal that leads to superlinear convergence. 
Numerical examples are presented to verify the theoretical results 
on the convergence behavior of the \psdid method for solving 
ill-conditioned Hermitian-definite generalized eigenvalue problems arising from 
electronic structure calculations.  
While rigorous and full-scale convergence proofs of 
preconditioned block steepest descent methods in practical use 
still largely eludes us, we believe the theoretical results presented 
in this paper sheds light on an improved understanding of the
convergence behavior of these block methods.

\vskip 2mm


\vskip2mm

\end{abstract}


\section{Introduction}\label{sec:intro}
We consider the Hermitian-definite generalized eigenvalue problem
\begin{equation}\label{eq:ghep}
Hu=\lambda Su,
\end{equation}
where $H$ and $S$ are $n$-by-$n$ Hermitian matrices and $S$ is 
positive-definite. The scalar $\lambda$ and nonzero vector $u$ 
satisfying \eqref{eq:ghep} are called
{\em eigenvalue} and {\em eigenvector}, respectively.
The pair $(\lambda,u)$ is called an eigenpair.
All eigenvalues of \eqref{eq:ghep} are known to be real.
Our task is to compute few smallest eigenvalues and the 
corresponding eigenvectors. We are particularly interested in solving 
the eigenvalue problem~\eqref{eq:ghep},  
where the matrices $H$ and $S$ are large and sparse, and 
there is no obvious gap between the eigenvalues of interest 
and the rest. Furthermore, $S$ is nearly singular 
and $H$ and $S$ share a near-nullspace. 
It is called an ill-conditioned generalized eigenvalue problem 
in \cite{fix1972algorithm}, a term we will adopt in this paper. 
The ill-conditioned generalized eigenvalue problem is considered to 
be an extremely challenging problem.\footnote{W. Kahan, 
Refining the general symmetric definite eigenproblem, poster presentation at
Householder Symposium XVIII 2011, 
available http://www.cs.berkeley.edu/$\sim$wkahan/HHXVIII.pdf}   

\medskip
Beside examples such as those cited in \cite{fix1972algorithm}, 
the ill-conditioned eigenvalue problem~\eqref{eq:ghep} arises from the
discretization of enriched Galerkin methods. The partition-of-unity
finite element (PUFE) method~\cite{melenk1996partition}, 
which falls within the class of enriched
Galerkin methods, is a promising approach 
in quantum-mechanical materials calculations, 
see \cite{cai2013hybrid} and references therein. 
In the PUFE method, physics-based basis functions are added to the classical 
finite element (polynomial basis) approximation, which affords the
method improved accuracy at reduced costs versus 
existing techniques. However, due to near linear-dependence
between the polynomial and enriched basis functions,
the system matrices that stem from such methods are ill-conditioned, and
share a large common near-nullspace. Furthermore, there is in 
general no clear gap between the eigenvalues that are sought and the rest. 
Another example of the ill-conditioned eigenvalue problem~\eqref{eq:ghep} arises
from modeling protein dynamics using 
normal-mode analysis \cite{levitt1985protein,nishikawa1987normal,brooks1995harmonic,levi:15}.  

\medskip

In this paper, we focus on a preconditioned
steepest descent with implicit deflation method, \psdid method in short, 
to solve the eigenvalue problems~\eqref{eq:ghep}. 
The basic idea of the \psdid method is simple. 
Denote all the eigenpairs of \eqref{eq:ghep} by $(\lambda_1, u_1)$,
$(\lambda_2,u_2)$, \dots, $(\lambda_n,u_n)$, and the eigenvalue and
eigenvector matrices by
$\Lambda=\diag(\lambda_1,\lambda_2,\dots,\lambda_n)$
and $U=[u_1 \; u_2\; \cdots\; u_n]$, respectively.
Assume that the eigenvalues $\{\lambda_i\}$  are in an ascending order
$\lambda_1 \le\lambda_2\le\cdots\le \lambda_n$.
The following variational principles are well-known, 
see \cite[p.99]{wilkinsonalgebraic} for example:
\begin{align}\label{eq:evew-i}
\lambda_i=\min_{U_{i-1}^HS z=0}\rho(z)
\quad \mbox{and} \quad u_i=\argmin_{U_{i-1}^HS z=0}\rho(z),
\end{align}
where $U_{i-1}=[u_1\; u_2\; \cdots \; u_{i-1}]$ and 
$\rho(z)$ is the Rayleigh quotient 
\begin{equation}\label{eq:rq}
\rho(z)=\frac{z^HHz}{z^H Sz}.
\end{equation}
On assuming that $U_{i-1}$ is known,
one can find the $i$th eigenpair by minimizing the
the Rayleigh quotient $\rho(z)$
with $z$ being $S$-orthogonal against $U_{i-1}$ under the algorithmic framework
of the preconditioned steepest descent minimization. 

\medskip 

The idea of computing the algebraically largest eigenvalue and its corresponding
eigenvector of \eqref{eq:ghep} (with $B = I$) using the steepest descent (SD)  
method dates back to early 1950s \cite{hestenes1951method} and 
\cite[Chap.7]{fadeev1963computational}.  
In \cite{longsine1980simultaneous}, block steepest descent (BSD)  
methods are proposed to compute several eigenpairs simultaneously.
The preconditioned steepest descent (PSD) method was  
introduced around late 1950s~\cite{samokish1958steepest,petryshyn1968eigenvalue}.
The block PSD (BPSD) methods have appeared in 
the literature, see \cite{bramble1996subspace,neymeyr2014iterative} 
and references therein.
Like the PSD method, the \psdid method studied in this paper computes 
one eigenpair at a time. To compute the $i$th eigenpair, the search 
subspace of \psdid
is {\em implicitly} orthogonalized against the previously computed 
$i-1$ eigenvectors. Furthermore, the preconditioner at each iteration is 
flexible (i.e, could change at every iteration) and can be indefinite, instead
of being fixed and positive definite as 
in~\cite{samokish1958steepest,bramble1996subspace,neymeyr2014iterative}.  

\medskip 

Over the past six decades, there has been significant work on 
the convergence analysis of the SD, PSD and BPSD methods.  
The convergence of the SD method to compute a single eigenpair 
is presented in \cite[Chap.7]{fadeev1963computational}. 
For the BSD method,  the convergence of the first eigenpair 
is presented in \cite{longsine1980simultaneous} and 
``ordered convergence'' for multiple eigenpairs is declared. 
The (nonasymptotic) rate of convergence of the PSD method is first studied 
in \cite{samokish1958steepest}, which later is proven to 
be sharp \cite{ovtchinnikov2006sharp}. 
A comprehensive review of the convergence estimates of the PSD method, 
is presented in \cite{bramble1996subspace}.
The theoretical proofs of the convergence of the BPSD method
have still {largely eluded us}, we refer the readers 
to \cite{bramble1996subspace,ovtchinnikov2006cluster}
and two recent papers \cite{neymeyr2014block,neymeyr2014iterative}. 
In this paper, we present two main results 
(Theorems~\ref{thm:psd2} and~\ref{thm:Samokish2}) 
on the convergence and nonasymptotic rate of convergence  
of the \psdid method. These results extend the classical 
ones due to Faddeev and Faddeeva \cite[sec.74]{fadeev1963computational} 
and Samokish \cite{samokish1958steepest} for the SD and PSD methods. 
We show that with the proper choice of the shift,  
the well-known {indefinite} shift-and-invert preconditioner
is a flexible and locally accelerated preconditioner, 
and is asymptotically optimal that 
leads to superlinearly converge  of the \psdid method. 
Numerical examples shows the superlinear convergence of the
\psdid method with locally accelerated preconditioners for solving 
ill-conditioned generalized eigenvalue problems~\eqref{eq:ghep} arising from 
full self-consistent electronic structure calculations. 

\medskip 

We would like to note that the main objective of this paper is to provide 
a rigorous convergence analysis of the \psdid method with flexible and 
locally accelerated preconditioners than to advocate the usage of the
\psdid method in practice.  
The BPSD methods~\cite{bramble1996subspace, neymeyr2014iterative} and 
a recent proposed locally accelerated BPSD (LABPSD) presented in our 
previous work~\cite{cai2013hybrid} have demonstrated their efficiency 
for finding several eigenpairs simultaneously.
While a rigorous and full-scale convergence proof of the 
the BPSD methods still largely eludes us, 
we believe the analysis of the \psdid method presented in this paper 
can shed light on an improved understanding of the 
convergence behavior of the BPSD methods 
such as the LAPBSD method~\cite{cai2013hybrid} for
solving the ill-conditioned generalized eigenvalue problem~\eqref{eq:ghep} 
arising from the PUFE simulation of electronic 
structure calculations.

\medskip 

The rest of this paper is organized as follows.
In section~\ref{sec:dpsd}, we present the \psdid method 
and discuss its basic properties.
In section~\ref{sec:analysis}, we provide a 
convergence proof and a nonasymptotic estimate of 
the convergence rate of the \psdid method.
An asymptotically optimal preconditioner is discussed 
in section~\ref{sec:opt}. Numerical examples 
to illustrate the theoretical results are presented in section~\ref{sec:numer}.
We close with some final remarks in section~\ref{sec:conclusion}.

\medskip 

In the spirit of reproducible research, 
Matlab scripts of an implementation of the \psdid method, and 
the data that used to generate numerical results 
presented in this paper can be obtained from the URL
http://dsec.pku.edu.cn/$\sim$yfcai/psdid.html.


\section{Algorithm} \label{sec:dpsd}
Assuming that $U_{i-1}$ is already known, by \eqref{eq:evew-i}, 
one can find the $i$th eigenpair by minimizing the Rayleigh quotient $\rho(z)$
with $z$ being $S$-orthogonal against $U_{i-1}$.
Specifically, let us denote by $(\lambda_{i;j},u_{i;j})$ 
the $j$th approximation of $(\lambda_i,u_i)$ 
and assume that 
\begin{equation}  \label{eq:jassume} 
U^H_{i-1}S u_{i;j} = 0, \quad
\|u_{i;j}\|_S=1 \quad \mbox{and} \quad 
\lambda_{i;j}=\rho(u_{i;j}).
\end{equation}
To compute the $(j+1)$st approximate eigenpair 
$(\lambda_{i;j+1}, u_{i;j+1})$, by the steepest descent approach,  
the steepest decreasing direction 
of $\rho(z)$ is opposite to the gradient of $\rho(z)$ at $z=u_{i;j}$: 
\[
\nabla\rho(u_{i;j})= 2(H-\lambda_{i;j} S)u_{i;j}=2r_{i;j}.
\]
Furthermore, to accelerate the convergence, we use the following 
preconditioned search vector
\begin{equation}\label{eq:pij}
p_{i;j}=-K_{i;j}r_{i;j}, 
\end{equation}
where $K_{i;j}$ is a preconditioner. 
By a Rayleigh-Ritz projection based implementation, 
the $(j+1)$st approximate eigenpair 
$(\lambda_{i;j+1}, u_{i;j+1})$ 
computed by the preconditioned steepest descent method 
is given by
\begin{align}\label{eq:liju}
(\lambda_{i;j+1}, u_{i;j+1}) 
= (\gamma_i, Z_j w_i),
\end{align}
where $(\gamma_i,w_i)$ is the $i$th eigenpair of 
the projected matrix pair $(H_R, S_R) = (Z^H_j H Z_j, Z^H_jSZ_j)$, 
$\|w_i\|_{S_R}=1$, and $Z_j=[U_{i-1}\; u_{i;j} \; p_{i;j}]$ is the basis
matrix of the projection subspace. 
Here we assume that $Z_j$ is of full column rank. 

\bigskip

Algorithm~\ref{alg:psdid} is a summary of the aforementioned procedure. 
Since the first eigenvectors $U_{i-1}$ are implicitly deflated in the 
Rayleigh-Ritz procedure, 
we call Algorithm~\ref{alg:psdid} a preconditioned steepest descent 
with implicit deflation, {\psdid} in short. 
We note that the preconditioner $K_{i;j}$ is flexible. 
It can be changed at each iteration. 
If the preconditioner is fixed as a uniform positive definite matrix, 
i.e., $K_{i;j} = K > 0$,  then Algorithm~\ref{alg:psdid} 
is the SIRQIT-G2 algorithm in \cite{longsine1980simultaneous} with $K=I$ 
and initial vectors $X^{(0)}=[U_{i-1}\, u_{i;0}]$, and is the BPSD 
method \cite{knyazev2003efficient} with initial vectors $[U_{i-1}\, u_{i;0}]$.

\begin{algorithm}
\caption{\psdid}  \label{alg:psdid}
\begin{algorithmic}[1]
\REQUIRE $U_{i-1}$ 
         and initial vector $u_{i;0}$  

\ENSURE Approximate eigenpair $(\lambda_i, u_i)$ of $(\lambda_i, u_i)$

\STATE $\lambda_{i;0}= \rho(u_{i;0})$

\FOR {$j=0,1,\ldots$, until convergence}
\STATE compute $r_{i;j}=Hu_{i;j} - \lambda_{i;j}Su_{i;j}$

\STATE precondition $p_{i;j} = -K_{i;j}r_{i;j}$
\STATE compute the $i$th eigenpair $(\gamma_i,w_i)$ of 
 $(H_R,S_R)=(Z^H_j H Z_j, Z^H_jS Z_j)$,  $Z_j=[U_{i-1}\; u_{i;j} \; p_{i;j}]$
\STATE update $\lambda_{i;j+1} = \gamma_i$ and $u_{i; j+1} = Z_j w_i$
\ENDFOR
\end{algorithmic}
\end{algorithm}

\medskip

If Algorithm~\ref{alg:psdid} does not breakdown, 
i.e., the matrices $Z_j$ on line 5 are full column rank for all $j$, 
then a sequence of approximate eigenpairs 
$\{(\lambda_{i;j},u_{i;j})\}_j$ are produced.
The following proposition gives basic properties of the sequence. 
In particular, if the initial vector $u_{i;0}$ does not satisfy
the assumption~\eqref{eq:jassume}, the first approximate vector 
$u_{i;1}$ computed by Algorithm~\ref{alg:psdid} will suffice. 

\begin{proposition} \label{prop:basicproperties}   
If $Z_j$ is of full column rank, then
\begin{itemize}
\item[(a)] $U^H_{i-1}S u_{i;j+1} = 0$.
\item[(b)] $\|u_{i;j+1}\|_S=1$.
\item[(c)] $\lambda_{i;j+1} \geq \lambda_i$.
\item[(d)] $\lambda_{i;j+1} \le \lambda_{i;j}$.
\end{itemize}
\end{proposition}
\begin{proof} 
Results (a) and (b) are verified by straightforward calculation. 
The result (c) follows from the inequality
\begin{equation*}
\lambda_{i;j+1}=\rho({u}_{i;j+1})\ge \min_{U_{i-1}^HSz=0}\rho(z)=\lambda_i.
\end{equation*}
Finally, the result (d) follows from the facts that 
\[
\lambda_{i;j+1}=\lambda_i(H_R, S_R)
=\min_{U_{i-1}^H S Z_j w=0 }\rho(Z_j w)
\le \rho(Z_j w)|_{w=e_i}=\rho(u_{i;j})=\lambda_{i;j},
\]
where $e_i$ is the $i$th column vector of identity matrix of order $i+1$.
\end{proof}

The following proposition shows that with the proper choice of 
the preconditioner $K_{i;j}$, the basis matrix 
$Z_j=[U_{i-1}\; u_{i;j} \; p_{i;j}]$
is of full column rank,
which implies that Algorithm~\ref{alg:psdid} does not breakdown.

\begin{proposition}\label{lem:fr}
If $r_{i;j}\ne 0$ and $K_{i;j}$ is chosen such that
\begin{equation}\label{k-con}
K^{c}_{i;j} :=(U_{i-1}^{c})^H S K_{i;j}SU_{i-1}^c>0,\end{equation}
then the basis matrix $Z_j$ is of full column rank. Here
$U_{i-1}^c$ is complementary eigenvector matrix of
$U_{i-1}$, i.e., $U_{i-1}^c= [u_{i}\; \cdots \;  u_n]$.
\end{proposition}
\begin{proof}
We prove  that $Z_j$ is of full column rank by showing that
\[ 
\det(H_{R}-\lambda_{i;j}S_{R}) = 
\det( Z^H_j( H  -\lambda_{i;j}S)Z_j ) \neq 0.
\]
First, it can be verified that the projected matrix pair
$(H_R, S_R)$ can be factorized as follows:
\begin{equation}  \label{eq:hr}
(H_R, S_R)  = L^{-1}\left( \begin{bmatrix} 
\Lambda_{i-1} & 0\\ 
0&  H_{\bot}\end{bmatrix}, 
\begin{bmatrix} 
I_{i-1} & 0\\ 0 & S_{\bot}\end{bmatrix}
\right) L^{-H},
\end{equation}
where 
\[
L=\begin{bmatrix}
I_{i-1} & 0 & 0\\
0      & 1 & 0\\
-p_{i;j}^H SU_{i-1} & 0 & 1
\end{bmatrix}, \quad   
H_{\bot} = Z_{\bot}^H HZ_{\bot}, \quad 
S_{\bot} = Z_{\bot}^H SZ_{\bot}, 
\]
and 
$Z_{\bot}=[u_{i;j} \; p_{\bot}]$ and
$p_{\bot}=U_{i-1}^c(U_{i-1}^c)^HSp_{i;j}$.
Consequently, we have
\begin{align}\label{eq:dethrsr}
\det(H_R-\lambda_{i;j} S_R)=\det(\Lambda_{i-1}-\lambda_{i;j} I) 
\det(H_{\bot}-\lambda_{i;j}S_{\bot}). 
\end{align}
By Proposition~\ref{prop:basicproperties}(c),
we have $\lambda_{i;j}\geq \lambda_i$. 
Since $r_{i;j}\ne 0$, $\lambda_{i;j}>\lambda_i$.
Hence, we conclude that 
\begin{equation}  \label{eq:det1} 
\det(\Lambda_{i-1}-\lambda_{i;j} I)\ne 0. 
\end{equation} 
Next we show that $\det(H_{\bot}-\lambda_{i;j}S_{\bot}) \neq 0$.
We first note that since $U_{i-1}^HSu_{i;j}=0$, 
there exists a vector $a$ such that $u_{i;j}=U_{i-1}^c a$.
Then it follows that
\begin{align}\label{eq:rij}
r_{i;j}
=(H-\lambda_{i;j}S)U_{i-1}^ca
=SU_{i-1}^c(\Lambda_{i-1}^c-\lambda_{i;j}I)a
=SU_{i-1}^c(U_{i-1}^c)^Hr_{i;j},
\end{align}
where $\Lambda_{i-1}^c = \diag(\lambda_{i},\ldots, \lambda_n)$.
Note that $(U_{i-1}^c)^Hr_{i;j}\ne 0$ since $r_{i;j}\ne 0$.
Furthermore, using \eqref{eq:rij} and \eqref{k-con}, we have
\begin{align}
\det(H_{\bot}-\lambda_{i;j}S_{\bot})
& =\det\begin{bmatrix}  0 & r_{i;j}^H p_{\bot}\\ p_{\bot}^H  r_{i;j}& p_{\bot}^H (H-\lambda_{i;j}S) p_{\bot}\end{bmatrix}\notag\\
& = -|p_{\bot}^H  r_{i;j}|^2\notag\\
& = -|r_{i;j}^H U_{i-1}^c (U_{i-1}^c)^H S K_{i;j} SU_{i-1}^c (U_{i-1}^c)^H r_{i;j}|^2\notag\\
& = -|r_{i;j}^H U_{i-1}^c {K}^c_{i;j} (U_{i-1}^c)^H r_{i;j}| < 0. \label{eq:det2} 
\end{align}

By \eqref{eq:dethrsr}, \eqref{eq:det1} and \eqref{eq:det2},  
we conclude that $H_{R}-\lambda_{i;j}S_{R}$ is nonsingular, 
which implies that $Z_j$ is of full column rank. 
\end{proof}

\begin{definition} \label{def:epdp} 
A preconditioner $K_{i;j}$ satisfying the condition \eqref{k-con} is
called an {\rm effectively positive definite preconditioner}.
\end{definition} 

We note that an effectively positive definite preconditioner $K_{i;j}$ with $i>1$ 
is not necessarily to be symmetric positive definite.
For example, for any $\lambda_1<\sigma <\lambda_i$ and 
$\sigma$ is not an eigenvalue of $(H,S)$, $K_{i;j}=(H-\sigma S)^{-1}$ is
effectively positive definite, although $K_{i;j}$ is indefinite.

\medskip 

If  the preconditioner $K_{i;j}$ is chosen such that      
the search vector $p_{i;j}=-K_{i;j}r_{i;j}$ satisfies
\begin{equation}\label{ideal-p}
U^HS({u}_{i;j}+p_{i;j})=\xi=(\xi_1,\xi_2,\dots,\xi_n)^H\quad
\mbox{with $\xi_i\ne 0$ and $\xi_j=0$ for $j>i$},
\end{equation}
then
\begin{align}
\lambda_{i;j+1}
& = \min_{U_{i-1}^H S Z_j w=0}\rho(Z_j w) \notag \\ 
& =\min_w \rho(U_{i-1}^c (U_{i-1}^c)^H SZ_jw)\notag \\
& =\min_v\rho(U_{i-1}^c (U_{i-1}^c)^H S[u_{i;j} \ p_{i;j}]v) \notag \\
& \le  \rho(U_{i-1}^c (U_{i-1}^c)^H S[u_{i;j} \ p_{i;j}]v)|_{v=[1 \; 1]^\top} \notag \\ 
& =\rho(U_{i-1}^c (U_{i-1}^c)^H S(u_{i;j}+ p_{i;j}))\notag\\
& =\rho(U_{i-1}^c (U_{i-1}^c)^H SU\xi) \notag \\ 
& =\rho(\xi_i u_i ) =\lambda_i. \label{eq:lambdaless} 
\end{align}
Therefore, combining the inequality \eqref{eq:lambdaless} 
and Proposition~\ref{prop:basicproperties}(c),
we have $\lambda_{i;j+1}=\lambda_i$.
In this case, we refer to $p_{i;j}$ satisfying the equation \eqref{ideal-p}
as an {\em ideal search direction}.
The notion of an ideal search direction not only helps assessing the quality 
of a preconditioned search direction,  but also tells the desired property 
for the solution of the preconditioning equation $p_{i;j}=-K_{i;j} r_{i;j}$.

\section{Convergence analysis}\label{sec:analysis}
In this section, we prove the convergence of the \psdid method and 
derive a nonasymptotic estimate of the convergence rate.
For brevity, we assume that for the desired 
$i$th eigenvalue $\lambda_i$, it satisfies
$\lambda_{i-1} < \lambda_{i} < \lambda_{i+1}$.
Otherwise by replacing $\lambda_{i+1}$ by the smallest eigenvalue
of $(H,S)$ which is larger than $\lambda_i$, all results in this section
still hold, the proofs are similar. 

\subsection{Convergence results}\label{sec:convergence}

Assume that the preconditioner $K_{i;j}$ is effectively positive definite,
then by Proposition~\ref{prop:basicproperties}(d) and \eqref{eq:det2},
we have
that $\lambda_{i;j+1}$ is strictly less than $\lambda_{i;j}$,
\begin{equation}\label{lamstrictless}
\lambda_{i;j+1}<\lambda_{i;j}.
\end{equation}
Furthermore, by Proposition~\ref{prop:basicproperties}(c) 
and \eqref{lamstrictless}, the approximate eigenvalue sequence 
$\{\lambda_{i;j}\}_j$ is a monotonically 
decreasing and is bounded below by $\lambda_i$, i.e.,
\begin{equation} \label{eq:lambdadecrease} 
\lambda_{i;0} > \lambda_{i;1} > \cdots > \lambda_{i;j}>\lambda_{i;j+1}
> \dots \ge \lambda_i.
\end{equation} 
Therefore, the sequence $\{\lambda_{i;j}\}_j$ must converge. 
Does it converge to the $i$th eigenvalue $\lambda_i$ of $(H,S)$? 
How about the corresponding $\{u_{i;j}\}_j$?
We will answer these questions in this subsection. 
First, we give the following lemma to quantify the 
difference between two consecutive approximates 
$\lambda_{i;j}$ and $\lambda_{i;j+1}$ of $\lambda_i$.

\begin{lemma}\label{thm:lam}
If $r_{i;j}\ne 0$ and the preconditioner $K_{i;j}$ is 
effectively positive definite, then
\begin{equation}\label{eq:lamdif}
\lambda_{i;j}-\lambda_{i;j+1} \ge \sqrt{g^2+\phi^2} - g,
\end{equation}
where 
$g=(\lambda_n-\lambda_i)/2$ and 
$\phi ={\|r_{i;j}\|_{S^{-1}} }/{\kappa(K^c_{i;j})}$,
$\kappa(K^c_{i;j})$ is the condition number of ${K}^c_{i;j}$ 
defined in \eqref{k-con}.
\end{lemma}
\begin{proof}
Let $H_{\bot}$ and $S_{\bot}$ be the matrices defined in \eqref{eq:hr}.  
By $S$-orthogonalizing $p_{i;j}$ against $[U_{i-1} \; u_{i;j}]$, 
the resulting vector $\hat{p}=(I-u_{i;j}u_{i;j}^HS)P_{i-1}p_{i;j}$ 
must be nonzero since $Z_j=[U_{i-1} \; u_{i;j}\; p_{i;j}]$ is of full 
column rank (Lemma~\ref{lem:fr}).
Therefore,  it holds that $\hat{p}^HS\hat{p}> 0$.
By straightforward calculations, we have 
\begin{align}
\det(H_{\bot}-\mu S_{\bot})
& = \det\left(
\begin{bmatrix}\lambda_{i;j}& u_{i;j}^HH\hat{p}\\ \hat{p}^HHu_{i;j} & \hat{p}^HH\hat{p}\end{bmatrix} -
\mu\begin{bmatrix} 1 & 0\\ 0 & \hat{p}^HS\hat{p}\end{bmatrix}\right) \notag \\
& = \hat{p}^HS\hat{p} (\lambda_{i;j}-\mu)
     (\rho(\hat{p})-\mu)-|u_{i;j}^HH\hat{p}|^2 \notag \\
& =  \hat{p}^HS\hat{p} \left[ (\lambda_{i;j}-\mu)^2 +
(\rho(\hat{p})-\lambda_{i;j})(\lambda_{i;j}-\mu) 
-\frac{ |u_{i;j}^HH\hat{p}|^2}{\hat{p}^HS\hat{p}}\right]. \label{eq:detHS} 
\end{align}

By the definition of $\lambda_{i;j+1}$ in~\eqref{eq:liju},
the identity~\eqref{eq:dethrsr}, 
we know that $\lambda_{i;j+1}$ is the smaller 
root of the quadratic polynomial \eqref{eq:detHS} of $\mu$.
In addition, by \eqref{eq:lambdadecrease}, 
we know that $\lambda_{i;j}-\lambda_{i;j+1}$ is positive.
Therefore $\lambda_{i;j}-\lambda_{i;j+1}$ is 
the positive root of the following quadratic equation in $t$:
\begin{equation*}
t^2 + (\rho(\hat{p})-\lambda_{i;j}) t 
- \frac{|u_{i;j}^HH\hat{p}|^2}{ \hat{p}^HS\hat{p}}=0.
\end{equation*}
Then it follows that 
\begin{align}\label{eq:tplus}
\lambda_{i;j}-\lambda_{i;j+1}
=-\frac{\rho(\hat{p})-\lambda_{i;j}}{2}
+\sqrt{\left(\frac{\rho(\hat{p})-\lambda_{i;j}}{2}\right)^2
+\frac{|u_{i;j}^HH\hat{p}|^2}{ \hat{p}^HS\hat{p}}}.
\end{align}
In what follows,  we give the estimates of 
the quantities $|\rho(\hat{p})-\lambda_{i;j}|$, 
$|u_{i;j}^HH\hat{p}|^2$ and $\hat{p}^HS\hat{p}$, respectively.

For the quantity $|\rho(\hat{p})-\lambda_{i;j}|$, 
using the fact that for any nonzero $z$ satisfying $U_{i-1}^HSz=0$,
it holds $\lambda_i\le\rho(z)\le\lambda_n$, then 
using $U_{i-1}^HS\hat{p}=0$ and $U_{i-1}^H S u_{i;j}=0$, we have  
\begin{equation} \label{eq:rhoest}  
0\le |\rho(\hat{p})-\lambda_{i;j}|\le \lambda_n-\lambda_i= 2g.
\end{equation}

For the quantity $|u_{i;j}^HH\hat{p}|^2$, we have
\begin{subequations}
\begin{align}
|u_{i;j}^HH\hat{p}|
& =|u_{i;j}^HH(I-u_{i;j}u_{i;j}^HS)U_{i-1}^c(U_{i-1}^c)^HSK_{i;j}r_{i;j}| 
           \label{uhp-1}\\
&=|\big[u_{i;j}^HH-\lambda_{i;j}u_{i;j}^HS\big]  U_{i-1}^c \big[(U_{i-1}^c)^HSK_{i;j}  SU_{i-1}^c \big] (U_{i-1}^c)^H  r_{i;j}|\label{uhp-2}
\\
&=|r_{i;j}^HU_{i-1}^c K^c_{i;j} (U_{i-1}^c)^Hr_{i;j}|\label{uhp-3}\\
&\ge  \lambda_{\min}(K^c_{i;j})\|(U_{i-1}^c)^Hr_{i;j}\|^2\label{uhp-4}\\
& = \lambda_{\min}(K^c_{i;j})\|r_{i;j}\|_{S^{-1}}^2\label{uhp-5},
\end{align}
\end{subequations}
where 
\eqref{uhp-1} uses  the definition of $\hat{p}$  and \eqref{eq:pij},
\eqref{uhp-2} uses the fact that 
$r_{i;j}=SU_{i-1}^c (U_{i-1}^c)^H r_{i;j}$, 
\eqref{uhp-3} and \eqref{uhp-4} use 
the definition of $K^c_{i;j}$ in \eqref{k-con} and the assumption 
that $K^c_{i;j}$ is symmetric positive definite, and 
\eqref{uhp-5} is based on the following calculations: 
\begin{align*}
\|(U_{i-1}^c)^Hr_{i;j}\|^2
&=r_{i;j}^H U_{i-1}^c(U_{i-1}^c)^Hr_{i;j} &\\
&=r_{i;j}^H U_{i-1}^c(U_{i-1}^c)^Hr_{i;j}+r_{i;j}^H U_{i-1}U_{i-1}^Hr_{i;j} &  (U_{i-1}^Hr_{i;j}=0)\\
&=r_{i;j}^H U U^H r_{i;j}&\\
&=r_{i;j}^H S^{-1}r_{i;j} & (UU^H=S^{-1})\\
&= \|r_{i;j}\|_{S^{-1}}^2. &
\end{align*}

For the quantity $\hat{p}^HS\hat{p}$, we have
\begin{align}
\hat{p}^HS\hat{p} 
& = (r_{i;j})^HK_{i;j} SU_{i-1}^c   \big[ (U_{i-1}^c)^H(I-Su_{i;j}u_{i;j}^H)
S (I-u_{i;j}u_{i;j}^HS) U_{i-1}^c \big]  (U_{i-1}^c)^H S K_{i;j} r_{i;j}\notag\\ 
& \le\| (U_{i-1}^c)^H(I-Su_{i;j}u_{i;j}^H)
S (I-u_{i;j}u_{i;j}^HS) U_{i-1}^c \| \|(U_{i-1}^c)^H S K_{i;j} r_{i;j} \|^2\notag\\
& \le \|(U_{i-1}^c)^H S K_{i;j} r_{i;j} \|^2 \notag \\ 
& = \|K^c_{i;j} (U_i^{c})^Hr_{i;j}\|^2 \notag\\ 
& \le \lambda_{\max}(K^c_{i;j})^2\|(U_{i-1}^c)^Hr_{i;j}\|^2\notag\\
& = \lambda_{\max}(K^c_{i;j})^2\|r_{i;j}\|_{S^{-1}}^2, \label{php-7}
\end{align}
where  
the second inequality use the fact that
\begin{align*}
\| &(U_{i-1}^c)^H(I-Su_{i;j}u_{i;j}^H) S (I-u_{i;j}u_{i;j}^HS) U_{i-1}^c \|\\
 = \| &(U_{i-1}^c)^HS^{\half}(I-S^\half u_{i;j}u_{i;j}^H S^\half)
(I-S^\half u_{i;j}u_{i;j}^HS^\half) S^\half U_{i-1}^c \| \\
\le \| &(I-S^\half u_{i;j}u_{i;j}^H S^\half)\|^2 \|S^\half U_{i-1}^c \|^2 \le 1.
\end{align*}

Finally, by \eqref{eq:rhoest}, \eqref{uhp-5} and 
\eqref{php-7}, it follows from \eqref{eq:tplus} that
\begin{align*}
\lambda_{i;j}-\lambda_{i;j+1}
\ge& -\frac{|\rho(\hat{p})-\lambda_{i;j}|}{2}+\sqrt{\bigg(\frac{\rho(\hat{p})-\lambda_{i;j}}{2}\bigg)^2 + \frac{|u_{i;j}^HH\hat{p}|^2}{ \hat{p}^HS\hat{p}}}\\
\ge & -\frac{\lambda_n-\lambda_i}{2} + \sqrt{\bigg(\frac{\lambda_n-\lambda_i}{2}\bigg)^2+\frac{\|r_{i;j}\|^2_{S^{-1}}}{\kappa^2(K^c_{i;j})}}\\
= &-g+\sqrt{g^2+\phi^2}.
\end{align*}
This completes the proof.
\end{proof}

We note that in \cite[Chap.7]{fadeev1963computational}, for the 
steepest descent method to compute the largest eigenvalue 
$\lambda_n$ of a Hermitian matrix, it shows that
\begin{equation*}
\lambda_{n;j+1}-\lambda_{n;j}\ge
\frac{\|r_{n;j}\|^2}{\lambda_n-\lambda_1}. 
\end{equation*}
Then it is established that  $\lambda_{n;j}$ converges to $\lambda_n$, and
$u_{n;j}$ converges to $u_n$ directionally.
Lemma~\ref{thm:lam} and the following theorem 
are generalizations that are not limited to the largest eigenpair,
and include the usage of flexible preconditioners. 

\begin{theorem}\label{thm:psd2}
If the initial estimate eigenvalue 
$\lambda_{i;0}$ satisfying $\lambda_i < \lambda_{i;0}<\lambda_{i+1}$, and
the flexible preconditioners $K_{i;j}$ are an effectively positive definite 
for all $j$ and $\mbox{sup}_j \kappa(K^c_{i;j})=q<\infty$, 
then the sequence $\{(\lambda_{i;j}, u_{i;j})\}_j$ generated by the 
\psdid method converges to the desired pair $(\lambda_{i}, u_{i})$, i.e.,
\begin{itemize} 
\item[(a)]  $\lim_{j\rightarrow \infty}\lambda_{i;j}=\lambda_i$.
\item[(b)] $\lim_{j\rightarrow \infty}\|r_{i;j}\|_{S^{-1}}=0$, namely $u_{i;j}$ 
converges to $u_i$ directionally.
\end{itemize} 
\end{theorem}
\begin{proof} To prove (a), we first notice 
that $\{\lambda_{i;j}\}_j$ is a monotonic decreasing sequence,
and is bounded by $\lambda_i$ from below.
So there exists a real number $\tilde{\lambda}_i$ such that 
$\lambda_{i;j} \rightarrow \tilde{\lambda}_i$ as 
$j\rightarrow \infty$.  Now we show by contradiction 
that $\tilde{\lambda}_i=\lambda_i$.
For any $u_{i;j}$ ($\|u_{i;j}\|_S=1$), we have 
\begin{align*}
\|r_{i;j}\|_{S^{-1}}
& =\|(H-\lambda_{i;j}S)u_{i;j}\|_{S^{-1}} \\ 
& \ge \|(H-\tilde{\lambda}_iS)u_{i;j}\|_{S^{-1}}-(\lambda_{i;j}-\tilde{\lambda}_i)\|Su_{i;j}\|_{S^{-1}}\\
& \ge \min_k{|\lambda_k-\tilde{\lambda}_i|}-(\lambda_{i;j}-\tilde{\lambda}_i).
\end{align*}
As $\lim_{j\rightarrow\infty}\lambda_{i;j}=\tilde{\lambda}_i$, 
there exists a $j_0$ such that for any $j\ge j_0$,
\[
\|r_{i;j}\|_{S^{-1}}> \half\min_k|\lambda_k-\tilde{\lambda}_i|.
\]
By defining $d(r,\kappa):=-g+\sqrt{g^2+(r/\kappa)^2}$,
it follows from Lemma~\ref{thm:lam} that 
for any $j\ge j_0$, it holds that
\begin{align*}
\lambda_{i;j}-\lambda_{i;j+1} 
\ge  d(\|r_{i;j}\|_{S^{-1}},\kappa_{i;j})
>   d(\min_k|\lambda_k-\tilde{\lambda}_i|/2, \kappa_{i;j}) 
 \ge   d(\min_k|\lambda_k-\tilde{\lambda}_i|/2,   q),
\end{align*}
which in the limit becomes
\[
0> d(\min_k|\lambda_k-\tilde{\lambda}_i|/2,   q).
\]
This is a contradiction to the fact that 
$d(\min_k|\lambda_k-\tilde{\lambda}_i|/2,   q)$ is a positive constant.

To prove (b), by using $\lim_{j\rightarrow \infty}\lambda_{i;j}
=\lambda_i$ and Lemma~\ref{thm:lam}, we have
\begin{align*}
\lim_{j\rightarrow\infty} d( \|r_{i;j}\|_{S^{-1}},q)
\le \lim_{j\rightarrow\infty} d( \|r_{i;j}\|_{S^{-1}},\kappa_{i;j})
\le \lim_{j\rightarrow\infty}(\lambda_{i;j}-\lambda_{i;j+1})=0.
\end{align*}
Consequently, $\lim_{j\rightarrow\infty}d(\|r_{i;j}\|_{S^{-1}},q)=0$,
which leads to $\lim_{j\rightarrow\infty}\|r_{i;j}\|_{S^{-1}}=0$
since $1\le q<\infty$.
\end{proof}
We note that 
in Theorem~\ref{thm:psd2}, without assuming $\lambda_{i;0}<\lambda_{i+1}$, by similar argument, we can conclude
$\{\lambda_{i;j}\}_j$ converges to an eigenvalue $\lambda_k$ for some $k\ge i$,
and $\{u_{i;j}\}_j$ directionally converges to the corresponding eigenvector $u_k$.

\subsection{Rate of convergence}\label{subsec:asymptotic}
Theorem~\ref{thm:psd2} concludes the convergence of the 
sequence $\{\lambda_{i;j}\}_j$,
what follows we derive a nonasymptotic estimate of 
the convergence rate of $\{\lambda_{i;j}\}_j$ based on 
the work of Samokish in 1958~\cite{samokish1958steepest}. 
We begin by recalling the following equalities 
for the projection matrix 
$P_{i-1} = I - U_{i-1}U^H_{i-1}S = {U}_{i-1}^c({U}_{i-1}^c)^HS$: 
\begin{subequations}
\begin{align}
P_{i-1} u_{i;j} & =u_{i;j}, \label{eq-1}\\
P_{i-1}^2 & =P_{i-1}, \label{eq-2}\\
P_{i-1}^H(H-\lambda_i S)& =(H-\lambda_i S)P_{i-1}, \label{eq-3}\\
P_{i-1}^HS & =SP_{i-1}, \label{eq-4}
\end{align}
\end{subequations}

First, we have the following lemma.
\begin{lemma}\label{lem:eigen-distribution}
Define
\begin{equation}
M=P_{i-1}^H (H-\lambda_i S)P_{i-1}
\end{equation}
and assume that $K_{i;j}$ is effectively positive definite.
\begin{itemize} 
\item[(a)] $M$ is positive semi-definite and 
$M = G G^H$, where $G=S{U}_i^c({\Lambda}_i^c-\lambda_i I)^{\half}$ 
is of full column rank.

\item[(b)] All eigenvalues of $G^HK_{i;j}G$ are positive.

\item[(c)] The eigenvalues of $K_{i;j}M$ are given by
\begin{equation}
\lambda(K_{i;j}M)=\{0_{[i]}\}\cup\lambda\big(G^HK_{i;j}G\big),
\end{equation}
where $0_{[i]}$ stands for the multiplicity $i$ of the number 0.
\end{itemize} 
\end{lemma}
\begin{proof}
(a) By the definitions of $M$ and $P_{i-1}$, it easy to see that
\begin{align*}
M& =S{U}_{i-1}^c({U}_{i-1}^c)^H(H-\lambda_i S){U}_{i-1}^c({U}_{i-1}^c)^HS\notag\\
& =S{U}_{i}^c({\Lambda}_{i}^c-\lambda_i I)({U}_{i}^c)^HS
= GG^H \ge 0,
\end{align*}
where $G=S{U}_i^c({\Lambda}_i^c-\lambda_i I)^{\half}$.  

(b) Direct calculation leads to
\begin{align*}
G^HK_{i;j} G
=({\Lambda}_i^c-\lambda_i I)^{\half} \widetilde{K}_{22}({\Lambda}_i^c-\lambda_i I)^{\half},
\end{align*}
where $\widetilde{K}_{22}$ is the trailing $(n-i)$-by-$(n-i)$ principal 
submatrix of $K^c_{i;j}$ by deleting its first $i-1$ rows and first $i-1$ columns.
Since $K_{i;j}$ is effectively positive definite, we know that
$K^c_{i;j}>0$ and  hence $\widetilde{K}_{22}>0$. 
Thus all eigenvalues of $G^HK_{i;j} G$ are positive.

(c) It follows that
\begin{align*}
\lambda(K_{i;j}M)=\lambda(K_{i;j} GG^H)=\{0_{[i]}\}\cup\lambda(G^HK_{i;j} G),
\end{align*}
where we use the well-known identity 
$\lambda(AB)=\{0_{[m-n]}\}\cup\lambda(BA)$ for 
$A\in\mathbb{C}^{m\times n}$ and $B\in\mathbb{C}^{n\times m}$ and $m\ge n$.
\end{proof}

We now give a nonasymptotic estimate of the 
convergence rate of \psdid (Algorithm~\ref{alg:psdid}).

\begin{theorem}\label{thm:Samokish2}
Let $\epsilon_{i;j}= \lambda_{i;j}-\lambda_i$ and 
$\lambda_{i;j}$ be {localized}, namely
\begin{equation}\label{asym-con}
\tau (\sqrt{\theta_{i;j}\epsilon_{i;j}}+\delta_{i;j}\epsilon_{i;j}) < 1,
\end{equation}
then 
\begin{equation}\label{ineq:Samokish2}
\epsilon_{i;j+1}
\le \Bigg[\frac{\Delta + \tau\sqrt{\theta_{i;j}\epsilon_{i;j}}}
{1-\tau (\sqrt{\theta_{i;j}\epsilon_{i;j}}+
\delta_{i;j}\epsilon_{i;j})} \Bigg]^2 \epsilon_{i;j}, 
\end{equation}
where 
$\theta_{i;j}= \|S^{\half}K_{i;j}MK_{i;j}S^{\half}\|$, 
$\delta_{i;j}= \|S^{\half}K_{i;j}S^{\half}\|$, 
$\Delta={(\Gamma - \gamma)}/{(\Gamma + \gamma)}$,
$\Gamma$ and $\gamma$ are
the largest and smallest positive eigenvalues of $K_{i;j}M$,
respectively, and $\tau={2}/{(\Gamma+\gamma)}$.
\end{theorem}
\begin{proof}
Recall $Z_{\bot}=[\, u_{i;j}\; P_{i-1}p_{i;j}\,]$ defined in~\eqref{eq:hr}. 
It is easy to see that by using $Z_{\bot}$, the $(j+1)$th approximate eigenpair 
$(\lambda_{i;j+1}, u_{i;j})$ can be written as 
\[ 
\lambda_{i;j+1} = \min_{v}\rho(Z_{\bot}v). 
\]
Considering a choice of the vector $v$ for the line search, we have
\[
\lambda_{i;j+1} = \min_{v}\rho(Z_{\bot}v)
\leq \rho(Z_{\bot}v)|_{v=[1 \; \tau]^{\top}} = \rho(z),
\]
where $z=Z_{\bot}[1\; \tau]^{\top}= u_{i;j}+\tau P_{i-1} p_{i;j}$. 
Consequently, we have
\begin{align} \label{ineq:samo}
\epsilon_{i;j+1} =  \lambda_{i;j+1}-\lambda_i
\le \rho(z)-\lambda_i=
\frac{z^H (H - \lambda_i S) z} {z^H S z}.
\end{align}
In the following, we provide estimates for the
numerator and denominator of the upper bound \eqref{ineq:samo}. 

For the {numerator} of the upper bound in \eqref{ineq:samo},  it follows that
\begin{subequations}\label{eq-zmz}
\begin{align}
 z^H(H-\lambda_i S)z
& = (u_{i;j}+\tau p_{i;j})^HP_{i-1}^H(H-\lambda_i S)P_{i-1}(u_{i;j}
    +\tau p_{i;j})\notag \\
& =\|u_{i;j}+\tau p_{i;j}\|^2_{M}\notag\\
&=\|u_{i;j}-\tau P_{i-1}K_{i;j}(H-\lambda_{i;j}S)u_{i;j}\|_M^2\notag\\
&=\|u_{i;j}-\tau P_{i-1}K_{i;j}[(H-\lambda_iS)-\epsilon_{i;j}S]u_{i;j}\|_M^2\notag\\
&=\|[I-\tau P_{i-1}K_{i;j}(H-\lambda_iS)]u_{i;j}
+\tau \epsilon_{i;j} P_{i-1}K_{i;j}Su_{i;j}\|_M^2\notag\\
&=\|[I-\tau P_{i-1}K_{i;j}P_{i-1}^H(H-\lambda_iS)P_{i-1}]u_{i;j}
+\tau \epsilon_{i;j} P_{i-1}K_{i;j}S u_{i;j}\|_M^2\label{eq:zmz1}\\
& \le \big(\|[I-\tau P_{i-1}K_{i;j} M]u_{i;j}\|_{M}
   + \tau \epsilon_{i;j}\|P_{i-1}K_{i;j} Su_{i;j}\|_{M}\big)^2\label{eq:zmz4},
\end{align}
\end{subequations}
where the equality \eqref{eq:zmz1} uses the 
identities \eqref{eq-1} and \eqref{eq-3}. 
The inequality \eqref{eq:zmz4} uses the triangular inequality  of the
vector norm  induced by the semi-positive definite matrix $M$. 
For the first term in \eqref{eq:zmz4}, using $M=GG^H$ and $G^HP_{i-1}=G^H$, 
where $G$ is defined in Lemma~\ref{lem:eigen-distribution},
we have
\begin{align}
\|[I-\tau P_{i-1}K_{i;j} M]u_{i;j}\|_{M}
&= \| G^H[I-\tau P_{i-1}K_{i;j} GG^H ]u_{i;j}\|\notag\\
&= \|(I -\tau G^HK_{i;j} G)(G^H u_{i;j})\|\notag\\
&\le \|(I -\tau G^H K_{i;j} G)\| \|u_{i;j}\|_M.\label{eq:samo-11}
\end{align}
Note that by Lemma~\ref{lem:eigen-distribution}, it yields that
\begin{align} 
\|(I -\tau G^H K_{i;j} G)\| & =\max_k|1-\tau\lambda_k(G^H K_{i;j} G)| \notag \\ 
& =\max\{|1-\tau\gamma|,|1-\tau\Gamma|\} \notag \\ 
& =\frac{\Gamma-\gamma}{\Gamma+\gamma}=\Delta.  \label{eq:samo-12}
\end{align} 
Consequently, we can rewrite \eqref{eq:samo-11} as
\begin{align}\label{numerator1}
\|[I-\tau P_{i-1}K_{i;j}M]u_{i;j}\|_{M}
\le \Delta\,\|u_{i;j}\|_M = \Delta \sqrt{\epsilon_{i;j}}.
\end{align}
For the second term in \eqref{eq:zmz4}:
\begin{subequations}
\label{numerator2}
\begin{align}
\|P_{i-1}K_{i;j} Su_{i;j}\|^2_M
& = {u_{i;j}^H SK_{i;j} P_{i-1}^H M P_{i-1} K_{i;j} S u_{i;j}} \notag\\
&\le \|S^{\half}K_{i;j}MK_{i;j}S^{\half}\| \|S^{\half}u_{i;j}\|^2\label{eq22}\\
&=\|S^{\half}K_{i;j}MK_{i;j}S^{\half}\|\label{eq23}\\
&=\theta_{i;j},\notag
\end{align}
\end{subequations}
where \eqref{eq22} uses \eqref{eq-2},
\eqref{eq23} uses the fact $\|u_{i;j}\|_S=1$.
Combining \eqref{numerator1} and \eqref{numerator2},  
an estimate of the {numerator} of the upper bound in \eqref{ineq:samo} is given by
\begin{align}
z^H(H-\lambda_i S)z &\leq
(\Delta+\tau\sqrt{\theta_{i;j}\epsilon_{i;j}})^2 \epsilon_{i;j}.\label{numerator}
\end{align}

For the {denominator} of the upper bound \eqref{ineq:samo}, we first note that
\begin{align}
z^HSz & =\|u_{i;j}+\tau P_{i-1} p_{i;j}\|_S^2  \notag \\ 
      & \ge (\|u_{i;j}\|_S - \tau\|P_{i-1} p_{i;j}\|_S)^2 \notag  \\ 
      & = (1- \tau\|P_{i-1} p_{i;j}\|_S)^2.  \label{ineq:zsz}
\end{align}
By calculations, we have the following upper bound for $\|P_{i-1} p\|_S$:
\begin{subequations}
\label{denominator}
\begin{align}
\|P_{i-1} p_{i;j}\|_S
&= \| P_{i-1} K_{i;j} (H-\lambda_{i;j} S)u_{i;j} \|_S \notag\\
&=\|P_{i-1} K_{i;j}(H-\lambda_i S)u_{i;j} -\epsilon_{i;j}P_{i-1} K_{i;j}Su_{i;j}\|_S\notag\\
&\le \|P_{i-1} K_{i;j}(H-\lambda_i S)u_{i;j}\|_S +\epsilon_{i;j} \|P_{i-1} K_{i;j}Su_{i;j}\|_S\notag\\
& = \|P_{i-1} K_{i;j}Mu_{i;j}\|_S +\epsilon_{i;j} \|P_{i-1} K_{i;j}Su_{i;j}\|_S\label{eq33}\\
& \le \|S^{\half}P_{i-1}S^{-\half}\|\|S^{\half} K_{i;j} M^{\half}\| \|M^{\half}u_{i;j}\|
  + \epsilon_{i;j} \|S^{\half}P_{i-1} S^{-\half}\| \|S^{\half}K_{i;j}S^{\half}\| \|S^\half u_{i;j}\|\notag\\
&\le\sqrt{\theta_{i;j}\epsilon_{i;j}} +\delta_{i;j}\epsilon_{i;j}, \label{eq34}
\end{align}
\end{subequations}
where the equality \eqref{eq33} uses \eqref{eq-1} and \eqref{eq-3},
the inequality \eqref{eq34} uses the fact $\|S^\half P_{i-1} S^{-\half}\|\le 1$.

By \eqref{ineq:zsz} and \eqref{eq34},  if
\[
\tau (\sqrt{\theta_{i;j}\epsilon_{i;j}}+\delta_{i;j}\epsilon_{i;j}) < 1,
\]
then the {denominator} of the upper bound \eqref{ineq:samo} satisfies
\begin{equation}\label{ineq:zsz2}
z^HSz \geq (1- \tau (\sqrt{\theta_{i;j}\epsilon_{i;j}}+\delta_{i;j}\epsilon_{i;j}))^2.
\end{equation}
By combining \eqref{ineq:samo}, \eqref{numerator} 
and \eqref{ineq:zsz2}, we derive the 
estimate \eqref{ineq:Samokish2}. This concludes the proof.
\end{proof}

Theorem~\ref{thm:Samokish2} indicates that 
if $\lambda_{i;j}$ is localized (i.e., \eqref{asym-con} is satisfied), 
and  $\Delta +\tau\sqrt{\theta_{i;j}\epsilon_{i;j}}\rightarrow 0$ 
as $j\rightarrow \infty$, then  the \psdid method converges 
{\em superlinearly}.  In this case, we may call that the preconditioner $K_{i;j}$ 
is {\em asymptotically optimal}. In next section, we will consider 
such a preconditioner.

\medskip

To end this section
we note that for the smallest eigenvalue $\lambda_1$, if 
the preconditioner $K_{i;j}$ is chosen to be fixed and positive definite,
i.e., $K_{i;j} = K > 0$,
one can verify that $\theta_{i;j}\le\Gamma\delta_{i;j}$.
Theorem~\ref{thm:Samokish2} becomes the classical 
Samokish's theorem~\cite{samokish1958steepest,ovtchinnikov2006sharp},
which remains asymptotically most accurate estimate of the convergence
rate of the PSD method and is proven to be {\em sharp} \cite{ovtchinnikov2006sharp}.
The proof of Theorem~\ref{thm:Samokish2} relies 
on the triangular inequality~\eqref{eq-zmz},
which is inspired by the proof of Samokish's theorem
presented in \cite{ovtchinnikov2006sharp}. However, 
the treatment of each term in \eqref{ineq:samo} needs to be handled 
diligently in order to accommodates the effects of the projection 
matrix $P_{i-1}$ and the flexible preconditioner $K_{i;j}$.

\section{An asymptotically optimal preconditioner}\label{sec:opt}

In this section, we consider the shift-and-invert preconditioner: 
\begin{equation}\label{eq:kbeta}
\widehat{K}_{i;j}=\big(H-\beta_{i;j} S\big)^{-1},
\end{equation}
where $\beta_{i;j}$ is the shift. 
The following theorem shows that with a proper choice of $\beta_{i;j}$, 
$\widehat{K}_{i;j}$ is asymptotically optimal and consequently, 
the \psdid method converges superlinearly.

\begin{theorem}\label{lem:requirement} 
Consider the shift
\begin{equation}\label{eq:beta}
\beta_{i;j} = \lambda_{i;j} - c \|r_{i;j}\|_{S^{-1}},
\end{equation}
where  the constant
$c= \inf_k {\sqrt{(\lambda_{i;k}-\lambda_i)(\lambda_{i+1}-\lambda_{i;k})}}/{\|r_{i;k}\|_{S^{-1}}}$.
If 

\begin{subequations}
\begin{equation} \label{requirement-c} 
 c >3\sqrt{\Delta_{i;j}} 
\end{equation} 
and 
\begin{equation} \label{requirement2-1}
0 <\Delta_{i;j}<\min \left\{\frac{\Delta_i^2}{4}, \frac{1}{10}\right\},  
\end{equation}
\end{subequations}
where $\Delta_i={(\lambda_i-\lambda_{i-1})}/{(\lambda_{i+1}-\lambda_i)}$ and 
$\Delta_{i;j}={(\lambda_{i;j}-\lambda_i)}/{(\lambda_{i+1}-\lambda_{i;j})}$. 
Then 
\begin{itemize} 
\item[(a)]
 $\beta_{i;j}<\lambda_i$ and 
    $\widehat{K}_{i;j}$ is effectively positive definite. 

\item[(b)] 
 $\beta_{i;j} \rightarrow \lambda_i$ as $j\rightarrow\infty$.

\item[(c)]  
The condition \eqref{asym-con} of Theorem~\ref{thm:Samokish2} is satisfied,
namely, $\lambda_{i;j}$ is {\em localized}.

\item[(d)] 
$\Delta+\tau\sqrt{\theta_{i;j}\epsilon_{i;j}} \rightarrow 0$
as $j\rightarrow\infty$.
\end{itemize} 
By (c) and (d), the preconditioner $\widehat{K}_{i;j}$ 
is asymptotically optimal. 
\end{theorem}
\begin{proof} ~
(a) By the condition \eqref{requirement2-1}, 
the relation $0<\Delta_{i;j}<0.1$ implies that 
$\lambda_i$ is the closest eigenvalue to $\lambda_{i;j}$.
Let $u_{i;j}=U_{i-1}^c a$ for some $a$, 
then $(\lambda_{i;j}, a)$ is an approximated eigenpair of $\Lambda_{i-1}^c$. 
Using the Kato-Temple inequality \cite{kato1950upper}, we have
\begin{equation} \label{eq:ktineq}
(\lambda_{i;j}-\lambda_i)(\lambda_{i+1}-\lambda_{i;j})\le \| (\Lambda_{i-1}^c-\lambda_{i;j} I) a\|_2
=\|r_{i;j}\|_{S^{-1}}^2.
\end{equation} 
Therefore, the result $\beta_{i;j} < \lambda_i$ is verified as follows: 
\begin{align*}
\beta_{i;j}-\lambda_i
& =\lambda_{i;j} - c \|r_{i;j}\|_{S^{-1}} - \lambda_i \\ 
& \le \lambda_{i;j} -\lambda_i -c\sqrt{(\lambda_{i;j}-\lambda_i)(\lambda_{i+1}-\lambda_{i;j})} \\ 
& =(\lambda_{i+1}-\lambda_{i;j})(\Delta_{i;j}-c\sqrt{\Delta_{i;j}})<0,
\end{align*}
where for the last inequality we used the condition~\eqref{requirement-c}.

The preconditioner $\widehat{K}_{i;j}$ is effectively positive definite since
\begin{equation}\label{eq:kijc}
\widehat{K}^c_{i;j}=(U_{i-1}^c)^HS\widehat{K}_{i;j}SU_{i-1}^c
=\diag\Big(\frac{1}{\lambda_i-\beta_{i;j}},\frac{1}{\lambda_{i+1}-\beta_{i;j}},
\dots,\frac{1}{\lambda_{n}-\beta_{i;j}}\Big)
\end{equation}
and $\beta_{i;j} <\lambda_i$.

(b) 
By Theorem~\ref{thm:psd2}, we have
\begin{equation} \label{eq:betatolambda}
\beta_{i;j} \rightarrow \lambda_i \quad  \mbox{as $j\rightarrow\infty$}.
\end{equation} 

(c) With the choice of $\beta_{i;j}$ in \eqref{eq:beta}, for $\theta_{i;j}$, we have
\begin{equation} \label{eq:theta}
\theta_{i;j}
= \|S^{\half} \widehat{K}_{i;j}M \widehat{K}_{i;j} S^{\half}\|
= \max_{i+1 \le k \le n} \frac{\lambda_k-\lambda_i}{(\lambda_k-\beta_{i;j})^2}
= \frac{\lambda_{i+1}-\lambda_i}{(\lambda_{i+1}-\beta_{i;j})^2},\\ 
\end{equation} 
where for the last equality, 
we only need to show that $f'(x)<0$ for $x\ge \lambda_{i+1}$, 
where $f(x)=\frac{x-\lambda_i}{(x-\beta_{i;j})^2}$.
By calculations, we have
\[
f'(x)=\frac{2\lambda_i -x-\beta_{i;j}}{(x-\beta_{i;j})^3}<0
\]
since $x-\beta_{i;j}>0$ and
\begin{align*}
2\lambda_i -x - \beta_{i;j}
&\le 2\lambda_i -\lambda_{i+1} -\lambda_{i;j}+c\|r_{i;j}\|_{S^{-1}}
<  \lambda_{i;j} -\lambda_{i+1} + \sqrt{(\lambda_{i;j}-\lambda_i)(\lambda_{i+1}-\lambda_{i;j})}\\
&= (\lambda_{i+1}-\lambda_{i;j})(-1+\sqrt{\Delta_{i;j}})<0.
\end{align*}
For $\delta_{i;j}$, we have
\begin{equation} \label{eq:delta} 
\delta_{i;j}
=\|S^{\half} \widehat{K}_{i;j} S^{\half}\|
=\frac{1}{\min_{1\le k \le n}|\lambda_k-\beta_{i;j}|}
=\frac{1}{\lambda_i-\beta_{i;j}},\\ 
\end{equation} 
where for the last equality, we only need to show that $\beta_{i;j}-\lambda_{i-1}>\lambda_i-\beta_{i;j}$,
which is equivalent to
\begin{align*}
2c\|r_{i;j}\|_{S^{-1}} < 2\lambda_{ij}-\lambda_i-\lambda_{i-1}.
\end{align*}
Notice that the right hand side of the above inequality is no less than $\lambda_i-\lambda_{i-1}$,
thus, we only need to show
\[
2c\|r_{i;j}\|_{S^{-1}}<\lambda_i-\lambda_{i-1}.
\]
By calculations, we have
\begin{align*}
\frac{\lambda_i-\lambda_{i-1}}{2c\|r_{i;j}\|_{S^{-1}}}
\ge  \frac{\lambda_i-\lambda_{i-1}}{ 2 \sqrt{(\lambda_{i;j}-\lambda_i)(\lambda_{i+1}-\lambda_{i;j})}}
>\frac{ \Delta_i}{2\sqrt{\Delta_{i;j}}} \ge 1.
\end{align*}
In addition, using Lemma~\ref{lem:eigen-distribution}(c) and \eqref{eq:kijc},
it is easy to see that
\begin{align} \label{eq:temp} 
\Gamma  
= \frac{\lambda_n-\lambda_i}{\lambda_{n}-\beta_{i;j}} 
\quad \mbox{and} \quad
\gamma  
= \frac{\lambda_{i+1}-\lambda_i}{\lambda_{i+1}-\beta_{i;j}}.
\end{align}
Then it follows that 
\begin{align*}
\tau & = 2/(\Gamma + \gamma) \leq 1/\gamma, \\
\tau\sqrt{\theta_{i;j}\epsilon_{i;j}} & 
\le\frac{1}{\gamma}\sqrt{\theta_{i;j}\epsilon_{i;j}}
= \frac{\lambda_{i+1}-\beta_{i;j}}{\lambda_{i+1}-\lambda_i}
\sqrt{\frac{\lambda_{i+1}-\lambda_i}{(\lambda_{i+1}-\beta_{i;j})^2}(\lambda_{i;j}-\lambda_{i})}
= \sqrt{\Delta_{i;j}},\\
\frac{1}{\gamma}&
= \frac{\lambda_{i+1} - \lambda_{i;j} +{c\|r_{i;j}\|_{S^{-1}}} }{\lambda_{i+1}-\lambda_i}
< \frac{\lambda_{i+1}-\lambda_{i;j}+\sqrt{(\lambda_{i;j}-\lambda_i)(\lambda_{i+1}-\lambda_{i;j})}}{\lambda_{i+1}-\lambda_{i;j}}
=1+\sqrt{\Delta_{i;j}},\\
\delta_{i;j}\epsilon_{i;j}&=\frac{\lambda_{i;j}-\lambda_i}{\lambda_i-\beta_{i;j} }
=\frac{\lambda_{i;j}-\lambda_i}{\lambda_i- \lambda_{i;j} +{c\|r_{i;j}\|_{S^{-1}}} }  
<\frac{\lambda_{i;j}-\lambda_i}{\lambda_i- \lambda_{i;j} +{3\sqrt{\Delta_{i;j}}\|r_{i;j}\|_{S^{-1}}} }\\
&\le\frac{\lambda_{i;j}-\lambda_i}{\lambda_i- \lambda_{i;j} +{3\sqrt{\Delta_{i;j}}    \sqrt{(\lambda_{i;j}-\lambda_i)(\lambda_{i+1}-\lambda_{i;j})}   } }=\frac{1}{-1+3}=\frac{1}{2}.
\end{align*} 
Therefore, 
\begin{align*}
\tau (\sqrt{\theta_{i;j}\epsilon_{i;j}}+\delta_{i;j}\epsilon_{i;j})
 \le  \sqrt{\Delta_{i;j}} +\frac{1+\sqrt{ \Delta_{i;j}}}{2}  < 1 .
\end{align*}
(d) By the expressions \eqref{eq:temp}  of $\Gamma$ and $\gamma$, we have
\begin{align*}
\Delta=\frac{\Gamma-\gamma}{\Gamma+\gamma}
&=\frac{(\lambda_n-\lambda_i)(\lambda_{i+1}-\beta_{i;j}) -
 (\lambda_{i+1}-\lambda_i)( \lambda_{n}-\beta_{i;j}) }
          {(\lambda_n-\lambda_i)(\lambda_{i+1}-\beta_{i;j})  +
           (\lambda_{i+1}-\lambda_i)( \lambda_{n}-\beta_{i;j})}\notag\\
&=\frac{(\lambda_n-\lambda_{i+1})(\lambda_{i}-\beta_{i;j}) }
          {(\lambda_n-\lambda_i)(\lambda_{i+1}-\beta_{i;j})  +
         (\lambda_{i+1}-\lambda_i)( \lambda_{n}-\beta_{i;j})}\notag\\
&< \frac{(\lambda_n-\lambda_{i+1})(\lambda_{i}-\beta_{i;j}) }{2(\lambda_n-\lambda_i)(\lambda_{i+1}-\lambda_i) }
< \frac{\lambda_{i}-\beta_{i;j}}{2(\lambda_{i+1}-\lambda_i)}, \label{ineq:beta}
\end{align*}
Combining the above estimates of $\Delta$, 
$\theta_{i;j}\epsilon_{i;j}$ and $\tau$, 
we have
\begin{equation} \label{eq:deltatau}
\Delta+\tau\sqrt{\theta_{i;j}\epsilon_{i;j}}
 < \frac{\lambda_{i}-\beta_{i;j}}{2(\lambda_{i+1}-\lambda_i)} +
   \sqrt{\Delta_{i;j}}. 
\end{equation}
By Theorem~\ref{thm:psd2}(a) and
the result (a) of this theorem, 
the upper bound of \eqref{eq:deltatau} converges to zero as  $j\rightarrow \infty$.
\end{proof}

Four remarks are in order. 

\begin{remark}{\rm  
By the definition of the constant $c$ in \eqref{eq:beta}, we have
\[
c\|r_{i;j}\|_{S^{-1}} \le 
\sqrt{(\lambda_{i+1}-\lambda_i)(\lambda_{i;j}-\lambda_i)} 
\]
and 
\begin{equation}\label{eq:betaapprox} 
\beta_{i;j}=\lambda_{i;j}+\mathcal{O}((\lambda_{i;j}-\lambda_i)^\half).
\end{equation} 
Therefore, in practice,  
we can replace the shift $\beta_{i;j}$ by $\lambda_{i;j}$, and use 
the preconditioner
\begin{equation}\label{eq:ki}
{\widetilde{K}}_{i;j} = (H-\lambda_{i;j}S)^{-1}. 
\end{equation}
We call the preconditioner
${\widetilde{K}}_{i;j}$ a {\em locally accelerated} preconditioner.
} \end{remark} 

\begin{remark}{\rm  
With the locally accelerated preconditioner $\widetilde{K}_{i;j}$,
the corresponding search vector 
$\widetilde{p}_{i;j}=-\widetilde{K}_{i;j}r_{i;j}$. 
A direct calculation gives rise to
\begin{align*}
U^HS(u_{i;j}+\widetilde{p}_{i;j})
&= \begin{bmatrix}
0 & \dots & 0 &
\frac{\lambda_{i;j}-\beta_{i;j}}{\lambda_{i}-\beta_{i;j}}a_{i} &\dots &
\frac{\lambda_{i;j}-\beta_{i;j}}{\lambda_{n}-\beta_{i;j}}a_{n} \end{bmatrix}^{\top},
\end{align*}
where we use the fact $U^HSu_{i;j}=a=[0,\dots,0,a_i,\dots,a_n]^{\top}$. 
Then by Theorem~\ref{lem:requirement}(b), 
we have $U^HS(u_{i;j}+\widetilde{p}_{i;j}) \rightarrow  e_i^{\top}$ 
as $j\rightarrow \infty$.
In the notion of an ideal search vector introduced at the end of 
section~\ref{sec:dpsd}, the search vector $\widetilde{p}_{i;j}$ 
is an asymptotically ideal search vector.
} \end{remark} 

\begin{remark}{\rm  
Before $\lambda_{i;j}$ is localized,
we can use a fixed preconditioner $K_{i;j} = K$ for all $j$.
An obvious choice is to set $K_{i;j}\equiv K_{\sigma}=(H-\sigma S)^{-1}$
for some $\sigma<\lambda_1$. $K_{\sigma}$ is symmetric positive definite and
can be regarded as a global preconditioner for the initial few iterations.
By the convergence of \psdid (Theorem~\ref{thm:psd2}), 
it is guaranteed that the sequence $\{\lambda_{i;j}\}_j$ is strictly 
monotonically decreasing, albeit the convergence may be slow before the locally
accelerated preconditioner $\widetilde{K}_{i;j}$ is applied, see 
the numerical illustration in section~\ref{sec:numer}. 
}\end{remark} 

\begin{remark}{\rm  
As we discussed in section~\ref{sec:intro}, 
we are particularly interested in solving
ill-conditioned Hermitian-definite generalized eigenvalue problem~\eqref{eq:ghep}
where $H$ and $S$ sharing a common near-nullspace $\mathcal{V}$,
whose dimension can be large.
If we set the preconditioner $K_{i;j}\equiv I$,  then
$K_{i;j}M=M=S{U}_{i}^c({\Lambda}_{i}^c-\lambda_i I)({U}_{i}^c)^HS$,
which has a near-nullspace $\mathcal{V}$, and a nullspace $\subspan(U_i)$.
As $\dim(\mathcal{V})>\dim(\subspan(U_i))$, $K_{i;j}M$ has 
very small positive eigenvalues.
Therefore, ${\Gamma}/{\gamma}\gg 1$, and $\Delta\approx 1$.
By Theorem~\ref{thm:Samokish2}, we know that the \psdid method
would converge linearly.
By a similar arguments, we can declare that for any well-conditioned preconditioner 
$K_{i;j}$, the \psdid method would also converge linearly. 
Therefore, in order to achieve the fast convergence, one has to 
apply an ill-conditioned preconditioner such as 
the locally accelerated preconditioner ${\widetilde{K}}_{i;j}$.
} \end{remark}

\section{Numerical examples}\label{sec:numer}
In this section, we use a MATLAB implementation 
for the \psdid method (Algorithm~\ref{alg:psdid}) with
locally accelerated preconditioners ${\widetilde{K}}_{i;j}$ 
defined in \eqref{eq:ki} to generate
two numerical examples to verify the convergence and the rate of 
convergence of the method. 
To illustrate the efficiency of the method, we focus on 
two ill-conditioned generalized eigenvalue problems~\eqref{eq:ghep} arising
from the PUFE approach to solve differential eigenvalue equations arsing
in quantum mechanics. 
Matlab scripts of the implementation of the \psdid method and
the data that used to generate numerical results
presented in this section can be obtained from the URL
http://dsec.pku.edu.cn/$\sim$yfcai/psdid.html.

\medskip 

To apply ${\widetilde{K}}_{i;j}$, we need to test the 
localization conditions \eqref{requirement-c} and \eqref{requirement2-1}
of the $j$th approximate eigenvalue $\lambda_{i;j}$.  
For the condition \eqref{requirement-c}, note that $c$ is a constant and 
$\Delta_{i;j}$ in limit is zero. 
Therefore, when the residual $r_{i;j}$ is sufficiently small, 
$\lambda_{i;j}$ is close enough to $\lambda_i$, 
then the condition \eqref{requirement-c} will be satisfied.
Therefore, the test of the condition \eqref{requirement-c} can be
replaced by the following residual test: 
\begin{equation} \label{con-2} 
\mbox{Res}[\lambda_{i;j},u_{i;j}]=
\frac{\|Hu_{i;j}-\lambda_{i;j}Su_{i;j}\|}
     {\|Hu_{i;j}\|+|\lambda_{i;j}| \|Su_{i;j}\|}  \leq \tau_1,
\end{equation}
where $\tau_1$ is some prescribed threshold, say $\tau_1 = 0.1$.
%

For the condition \eqref{requirement2-1},
we need the estimates of eigenvalues $\lambda_i$ and $\lambda_{i+1}$ 
to approximate the quantities 
$\Delta_i = {(\lambda_{i}-{\lambda}_{i-1})}/{({\lambda}_{i+1} - \lambda_{i})}$
and 
$\Delta_{i;j} = {(\lambda_{i;j}-{\lambda}_{i})}/{({\lambda}_{i+1} - \lambda_{i;j})}$.
For $\Delta_i$,
it is natural to take  
the $j$th approximates 
${\lambda}_{i;j}$ and ${\lambda}_{i+1;j}$
of ${\lambda}_{i}$ and ${\lambda}_{i+1}$ respectively and yields the following estimate of $\Delta_i$ 
\[
\Delta_i \approx 
\widehat\Delta_i = 
\frac{\lambda_{i;j}-{\lambda}_{i-1}}{{\lambda}_{i+1;j} - \lambda_{i;j}}.
\]
For $\Delta_{i;j}$,
if we simply use $\lambda_{i;j}$ to estimate $\lambda_i$,  
then it leads to $\Delta_{i;j}=0$. 
This violates the condition \eqref{requirement2-1}.
A better estimate of $\lambda_i$ is 
to use the {linear extrapolation} 
$\widehat{\lambda}_i=2\lambda_{i;j}-\lambda_{i;j-1}$
of $\lambda_{i;j-1}$ and $\lambda_{i;j}$
for $j > 1$. Note that when $j=1$, all approximated eigenvalues 
are assumed to be not localized. 
Then it yields the following estimate of $\Delta_{i;j}$: 
\[
\Delta_{i;j} \approx 
\widehat\Delta_{i;j} = \frac{\lambda_{i;j}-\widehat{\lambda}_{i}}{{\lambda}_{i+1;j}-\lambda_{i;j}}
= \frac{\lambda_{i;j-1}-\lambda_{i;j}}{\lambda_{i+1;j} - \lambda_{i;j}}. 
\] 
In order to estimate ${\lambda}_{i+1}$, the Rayleigh-Ritz projection subspace 
in \psdid is spanned by the columns of 
$Z=[U_{i-1}\; u_{i;j}\; \dots\; u_{i+\ell;j}\; p_{i;j}]$ for some $\ell > 1$.
In this case, the \psdid method will also compute $\lambda_{i+1;j}$,
which can be used to approximate $\lambda_{i+1}$.

By the estimates $\widehat\Delta_{i}$ and $\widehat\Delta_{i;j}$, 
the localization condition \eqref{requirement2-1} of the 
$j$th approximate eigenvalue $\lambda_{i;j}$ of $\lambda_i$ 
can be verified by the following condition 
\begin{equation}  \label{con-1}
\widehat\Delta_{i;j} <  
\min\left\{\frac{1}{4} \widehat\Delta^2_i, 0.1 \right\} \equiv  \tau_2.  
\end{equation} 
Note that for computing the smallest eigenvalue $\lambda_1$, 
we let the initial approximate $\lambda_{0,j} = \sigma$ 
for some $\sigma < \lambda_1$.  Here $\sigma$ is a user given parameter or 
a lower bound of $\lambda_1$, 
say $\lambda_{1;j}-\|r_{1;j}\|_{S^{-1}}\approx \lambda_{1;j}-\|r_{1;j}\|$.

\medskip 

We use the preconditioned MINRES \cite{paige1975solution} to compute
the preconditioned search vector 
\begin{equation}\label{eq:pre}
p_{i;j}=- {\widetilde{K}}_{i;j} r_{i;j}
= -(H-\lambda_{i;j} S)^{-1} r_{i;j}.
\end{equation}
In practice, the vector $p_{i;j}$ is just needed to be computed
approximately such that
\begin{equation}\label{pij-con}
\|(H-\lambda_{i;j})p_{i;j}+r_{i;j}\| \le \eta_{i;j}\|r_{i;j}\|,
\end{equation}
where $\eta_{i;j}<1$ is a parameter. In our numerical experiments,
the preconditioner of the MINRES is $S^{-1}$, 
$\eta_{i;j}=\mbox{Res}[\lambda_{i;j},u_{i;j}]$,
and the maximum number of MINRES iterations is set to be 200.



%

\medskip
 
All numerical experiments are
performed on a quad-core $\text{Intel}^{\tiny\textregistered} \text{ Xeon}^{\tiny\textregistered}$ Processor E5-2643 running at 3.30GHz
with 31.3GB RAM, machine epsilon $\varepsilon\approx 2.2\times 10^{-16}$.

\begin{example} \label{eg:1dho} {\rm 
Consider the following Schr\"{o}dinger equation for 
a one-dimensional harmonic oscillator: 
\begin{align}\label{eq:oscillator}
-\half \psi''(x)+\half x^2 \psi(x) = 
E \psi(x), \qquad -L\le x\le L,\qquad \psi(-L)=\psi(L)=0,
\end{align}
where $E$ is the energy, $\psi(x)$ is the wavefunction.
This is a well-known model system 
in quantum mechanics \cite{liboff2003introductory,griffiths2004introduction}. 
If $L=\infty$, the eigenvalues 
of the equation \eqref{eq:oscillator} are $\lambda_i=i-0.5$ and 
the corresponding eigenfunctions are $\psi_i(x)=H_i(x)e^{-0.5x^2}$, 
where $H_i(x)$ is the $i$th order Hermite polynomial \cite[Chap. 18]{olver2010nist}.

\medskip

For numerical experiments, 
we set $L=10$ since $\psi_i(x)$ is numerically zero for $|x|>10$.
We discretize the equation \eqref{eq:oscillator} by linear finite element (FE), 
cubic FE and partition of unit FE (PUFE) \cite{melenk1996partition}, respectively.
In all three cases, the eigenfunction $\psi(x)$ is approximated by 
\begin{equation} \label{eq:pufebasis}
\psi^h(x) =\sum_{i} c_i \phi_i (x)+
\sum_{\alpha}\sum_{j} c_{j\alpha}\phi_j^{PU}(x)\tilde{\psi}_{\alpha}(x)
\equiv \sum_{k=1}^n u_k \Phi_k(x), 
\end{equation} 
where $\phi_i(x)$ are the FE basis functions,
$\phi_j^{PU}$ are the FE basis function to form 
enriched basis functions, $\tilde{\psi}_{\alpha}(x)$ are enrichment functions, 
and $c_i$ and $c_{j\alpha}$ are coefficients.
The enrichment term vanishes in the linear and cubic FE cases.
In our numerical experiments, the interval $[-10,10]$ is divided uniformly,
and for PUFE, 
$\phi_i(x)$, $\phi_j^{PU}(x)$ are chosen to be cubic and linear, respectively,
and $\tilde{\psi}_{\alpha}(x)=e^{-0.4x^2}$ for $x\in[-5,5]$ and zero elsewhere.

\medskip

Converting \eqref{eq:oscillator} into its weak form, and 
using $\Phi_i$ as  the test functions,
we obtain an algebraic generalized eigenvalue problem~\eqref{eq:ghep}, 
where $u=[u_1\, u_2\, \dots\, u_n]$ and $(i,j)$ elements 
$h_{ij}$ and $s_{ij}$ of $H$ and $S$ are given by 
\[
h_{ij}=\int_{-10}^{10}(\Phi'(x)\Phi'_j(x)+\frac12 x^2\Phi_i(x)\Phi_j(x))dx 
\quad \mbox{and} \quad 
s_{ij}=\int_{-10}^{10} \Phi_i(x)\Phi_j(x)dx, 
\] 
respectively. 
The left plot of Figure~\ref{fig:1dh} 
shows the errors of the sums of the four 
smallest eigenvalues of $(H,S)$ with respect to the number of FEs of 
three different finite element discretizations. 
The matrix sizes of linear FE are $7, 15, 31, 63, 127, 255$ and $511$.
The matrix sizes for the cubic FE are $23, 47, 95$ and $191$.
The matrix sizes for the PUFE are $28, 56$ $112$.
By the plot, we can see that to achieve the same accuracy, 
the matrix sizes of the PUFE are much smaller.
However, the condition numbers of PUFE matrices $H$, $S$ 
are large; $(\kappa_2(H), \kappa_2(S)) = 
(3.0\times 10^6, 5.0\times 10^6),
(6.5\times 10^8, 2.7\times 10^9),
(8.8\times 10^{10}, 8.1\times 10^{11})$, respectively.

\medskip 

For demonstrating the convergence behavior of \psdid, 
let us compute $m=4$ smallest eigenvalues of 
the PUFE matrices $H$ and $S$ of order $n=112$, which corresponds 
to the mesh size $h=2L/32$. The matrices $H$ and $S$ are ill-conditioned, 
$(\kappa_2(H), \kappa_2(S)) = (8.8\times 10^{10}, 8.1\times 10^{11})$.
Furthermore, $H$ and $S$ share a common near-nullspace, namely
there exists a subspace $\subspan(V)$ of dimension 17 such that 
$\|H V\| = \|SV\| = O(10^{-5})$.
To compute $4$ smallest eigenpairs,
we run the \psdid algorithm for $i=1,2,3,4$ with
$\ell=4$.  The accuracy threshold of computed eigenvalues 
is $\tau_{\rm eig}=10^{-9}$.
$\tau_1=0.1$ is used for the residual test \eqref{con-2}.

\medskip

The right plot of Figure~\ref{fig:1dh} shows  
the convergence history in the relative residuals 
$\mbox{Res}[\lambda_{i;j},u_{i;j}]$ of the \psdid method 
for computing four smallest eigenvalues. 
The localization (i.e., the conditions \eqref{con-2} and \eqref{con-1} are 
satisfied) of the $j$ approximate eigenpair $(\lambda_{i;j},u_{i;j})$ for
computing the $i$th eigenvalue $\lambda_i$ is marked by ``+'' sign. 
The locally accelerated preconditioner 
${\widetilde{K}}_{i;j} = (H-\lambda_{i;j} S)^{-1}$  
is used once $\lambda_{i;j}$ is localized.
As Theorem~\ref{lem:requirement} predicts, 
the locally accelerated preconditioner ${\widetilde{K}}_{i;j}$ 
is asymptotically optimal and leads to superlinear convergence 
of the \psdid algorithm. 

\begin{figure}
\begin{center}
\includegraphics[width=0.45\textwidth]{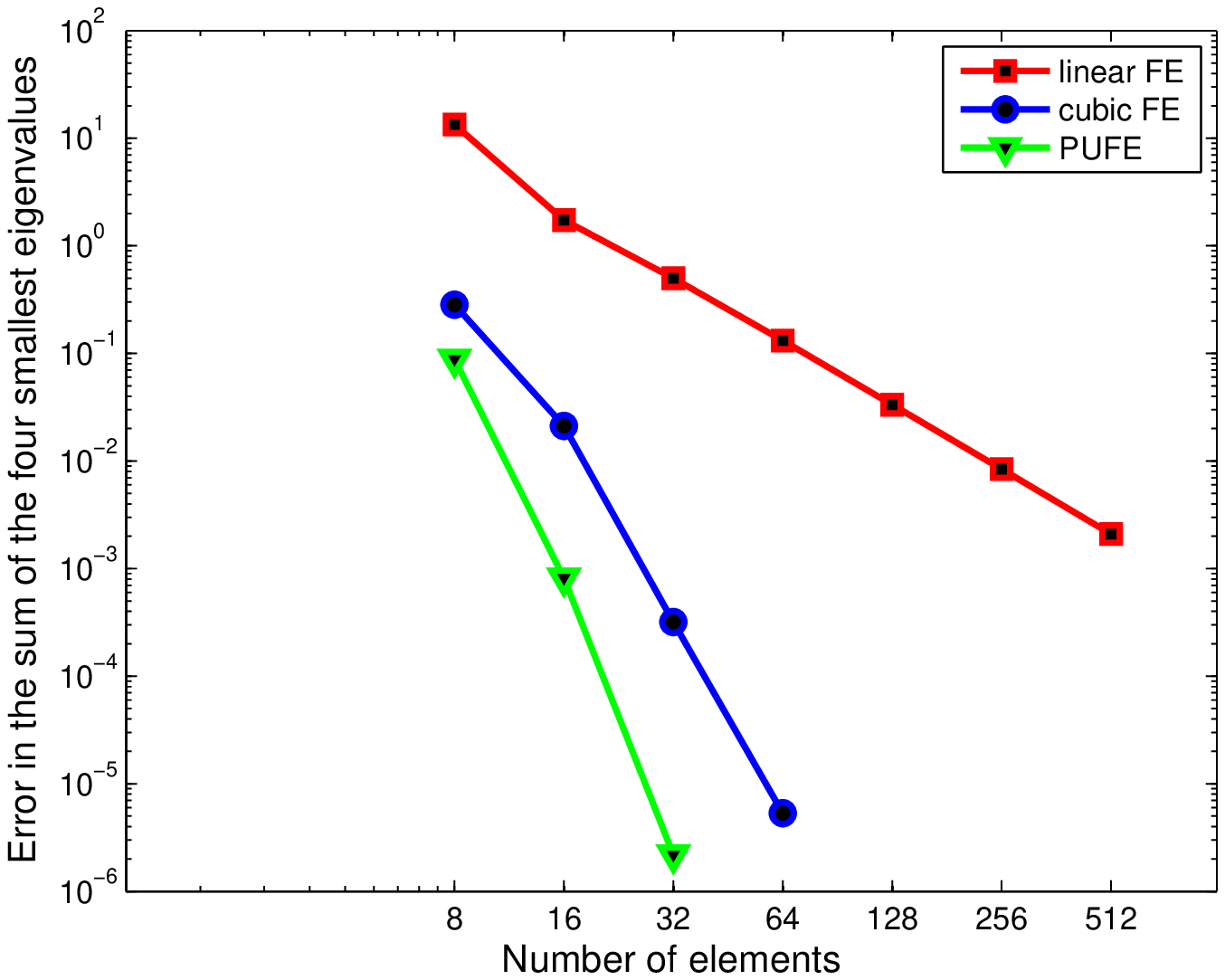}
\includegraphics[width=0.45\textwidth]{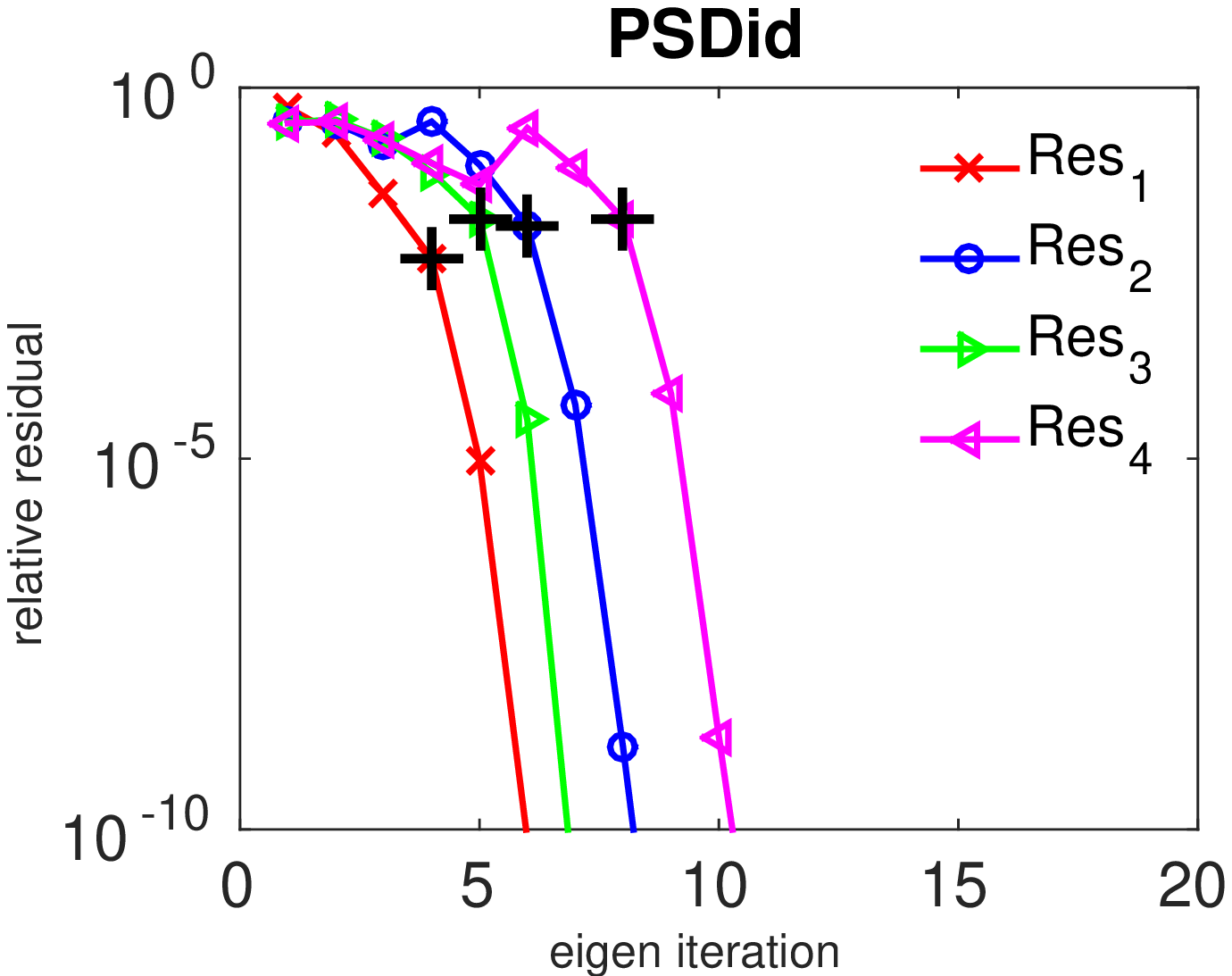}
\end{center}
\caption{Left: error of the sum of four smallest eigenvalues of 
$(H,S)$ with respect to the number of FEs of three different FE discretizatoins 
in Example~\ref{eg:1dho}. 
Right: convergence of the \psdid method for computing four smallest eigenvalues. 
} \label{fig:1dh} 
\end{figure} 

} \end{example} 

\begin{example}\label{eg:ceal} {\rm   
The Hermitian-definite generalized eigenvalue problem \eqref{eq:ghep} is a 
computational kernel in quantum mechanical methods employing a 
nonorthogonal basis for {\em ab initio} three-dimensional electronic 
structure calculations, see~\cite{cai2013hybrid} and references therein.
In this example, we select a sequence of eigenproblems produced by the PUFE 
method for a self-consistent pseudopotential density functional 
calculation for metallic, triclinic CeAl~\cite{sukumar2009classical,pask2012linear,pask2011partition}.  The Brillouin zone 
is sampled at two ${\bf k}$-points: 
${\bf k} = (0.00, \, 0.00,\,  0.00)$ and
${\bf k} = (0.12, \, -0.24,\, 0.37)$. The PUFE approximation for the 
wavefunction is of the form given in the equation~\eqref{eq:pufebasis}  
and we apply a standard Galerkin procedure to set up
the discrete system matrices.  The unit cell is a triclinic box,
with atoms displaced from ideal positions. The primitive lattice 
vectors and the position of the atomic centers are
\[ 
{\bf a}_1 = a(1.00 \quad 0.02 \quad -0.04), \quad
{\bf a}_2 = a(0.01 \quad 0.98 \quad  0.03), \quad
{\bf a}_3 = a(0.03 \quad -0.06 \quad 1.09)
\] 
and
\[ 
{\bf \tau}_{\rm{Ce}} = a(0.01 \quad 0.02 \quad 0.03), \quad
{\bf \tau}_{\rm{Al}} = a(0.51 \quad 0.47 \quad 0.55),
\] 
with lattice parameter $a = 5.75$ bohr.
Since Ce has a full complement of $s$, $p$, $d$, and $f$ states in 
valence, it requires 17 enrichment functions to span the occupied 
space. The near-dependencies between the finite element basis
functions and the enriched basis functions lead to an 
ill-conditioned generalized eigenvalue problem~\eqref{eq:ghep}. 

In this numerical example, the matrix size of $H$ and $S$ 
is $n=7\times 8^3 + 1752=5336$.  
Both $H$ and $S$ are ill conditioned and their condition numbers are 
$(\kappa_2(H), \kappa_2(S)) = (1.1641\times 10^{10}, 2.5731\times 10^{11})$.
Furthermore, $H$ and $S$ share a common near-nullspace $\subspan(V)$ of 
dimension 1000 such that $\|H V\| = \|SV\| = O(10^{-4})$, 
where $V$ is orthonormal. This is an extremely ill-conditioned eigenvalue problem. 
Figure~\ref{fig:CeAl} shows 
the convergence history of the \psdid method for computing
four smallest eigenvalues. As in Figure~\ref{fig:1dh}, 
the localization of the $j$ approximate eigenpair $(\lambda_{i;j},u_{i;j})$ 
is marked by ``+'' sign.  
Once $\lambda_{i;j}$ is localized, 
the locally accelerated preconditioner
${\widetilde{K}}_{i;j} = (H-\lambda_{i;j} S)^{-1}$
is used. 
Again, as Theorem~\ref{lem:requirement} predicts,
the locally accelerated preconditioner ${\widetilde{K}}_{i;j}$
leads to superlinear convergence of the \psdid algorithm.

\begin{figure}
\begin{center}
\includegraphics[width=0.5\textwidth]{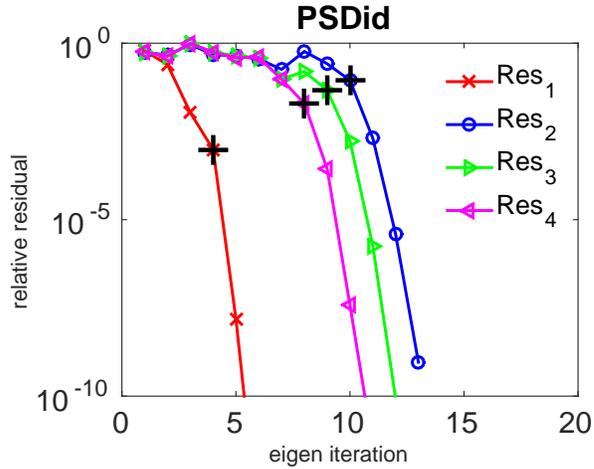}
\end{center}
\caption{
Convergence of the \psdid method for computing four smallest eigenvalues
of the CeAl matrix pair described in Example~\ref{eg:ceal}.
} \label{fig:CeAl} 
\end{figure} 

} \end{example}

\section{Conclusion}\label{sec:conclusion}
In this paper, we proved the convergence of 
the \psdid method, and derived a nonasymptotic estimate of 
the rate of convergence of the method.  
We show that with the proper choice of the shift,
the indefinite shift-and-invert preconditioner
is a locally accelerated preconditioner and 
leads to superlinear convergence.
Two numerical examples are presented to verify the theoretical results
on the convergence behavior of the \psdid method for solving
ill-conditioned Hermitian-definite generalized eigenvalue problems.

\bibliographystyle{abbrv} 

\end{document}